\documentclass[reqno]{amsart}

\usepackage[numbers]{natbib}
\usepackage{amssymb,mathtools,float,caption,wrapfig,hyperref,graphicx}
\usepackage[cal=boondox]{mathalfa}
\usepackage[abs]{overpic}
\hypersetup{colorlinks,linkcolor={blue},citecolor={blue},urlcolor={blue}}
\graphicspath{ {images/} }

\newtheorem{theorem}{Theorem}[section]

\newtheorem{lemma}[theorem]{Lemma}

\newtheorem{corollary}[theorem]{Corollary}

\theoremstyle{definition}
\newtheorem{definition}[theorem]{Definition}

\numberwithin{equation}{section}

\newcommand{\bZ}{\mathbb{Z}}

\newcommand{\bH}{\mathbb{H}}
\newcommand{\bC}{\mathbb{C}}
\newcommand{\cO}{\mathcal{O}}

\newcommand{\cF}{\mathcal{F}}

\begin{document}

\title[Projective elementary groups]{Fundamental polyhedra of projective elementary groups}

\thanks{The author is grateful to John Cremona for providing the data used to make Figure \ref{fig:1}.}

\author{Daniel E. Martin}
\address{Department of Mathematics, University of California, Davis, CA, USA}
\email{dmartin@math.ucdavis.edu}

\subjclass[2020]{Primary: 11R11, 30F40, 20F65. Secondary: 20G30, 52C05, 11A05.}

\keywords{Bianchi group, elementary matrices, fundamental polyhedron, Ford domain, imaginary quadratic, hyperbolic space, non-Euclidean}

\date{\today}

\begin{abstract}For $\cO$ an imaginary quadratic ring, we compute a fundamental polyhedron of $\text{PE}_2(\cO)$, the projective elementary subgroup of $\text{PSL}_2(\cO)$. This allows for new, simplified proofs of theorems of Cohn, Nica, Fine, and Frohman. Namely, we obtain a presentation for $\text{PE}_2(\cO)$, show that it has infinite-index and is its own normalizer in $\text{PSL}_2(\cO)$, and split $\text{PSL}_2(\cO)$ into a free product with amalgamation that has $\text{PE}_2(\cO)$ as one of its factors.\end{abstract}

\maketitle

\section{Introduction}\label{sec:1}

Given an order $\cO$ in an imaginary quadratic field $K$, let $\text{PE}_2(\cO)$ denote the \textit{projective elementary group} generated by the elementary (diagonal, triangular, or permutation) matrices in $\text{PSL}_2(\cO)$. The action of $\text{PE}_2(\cO)$ on the upper half-space model of hyperbolic 3-space, $\bH=\{(\zeta,t)\,|\,\zeta\in\bC,t\in(0,\infty)\}$, admits a fundamental polyhedron---a polyhedron whose orbit under $\text{PE}_2(\cO)$ tessellates $\bH$. (See Sections 1.1, 2.2, and 7.3 of \cite{elstrodt} for background.) Our aim is to compute one in particular:

\begin{definition}\label{def:bianchi}Given a fundamental polygon $F\subset\bC$ for $\cO$, the \emph{Ford domain} of $\Gamma\leq\text{PSL}_2(\cO)$ is $$\cF=\{P\in\bH\,|\,\zeta(P)\in \overline{F},\,t(P)\geq t(gP)\text{ for any }g\in\Gamma\},$$ where $\zeta(P)$ and $t(P)$ denote the first and second coordinates of $P$.\end{definition}

The image of $P\in\bH$ under some $g\in\text{PSL}_2(\cO)$ has second coordinate \begin{equation}\label{eq:1}t(gP)=\frac{t(P)}{|\alpha-\beta\zeta(P)|^2+|\beta|^2t(P)^2},\hspace{1cm}g=\begin{bmatrix}\delta & -\gamma\\ -\beta & \alpha\end{bmatrix}.\end{equation} Therefore if $\beta=0$ and $\alpha$ is a unit, $t(P)=t(gP)$ for all $P\in\bH$. In this case, assuming $\alpha=\pm 1$, upper triangular matrices shift $\cF$ just as they shift its projection $\overline{F}$. If $\beta\neq 0$, (\ref{eq:1}) shows that points $P$ satisfying $t(P)=t(gP)$ form an open Euclidean hemisphere centered at $\alpha/\beta\in\partial\hspace{0.03cm}\bH$ with radius $1/|\beta|$. We call this the \textit{isometric hemisphere} of $g$, denoted $S_g$. It follows from (\ref{eq:1}) that $t(P)\geq t(gP)$ if and only if $P$ is outside $S_g$. Thus, given a choice of $F$, the Ford domain of some $\Gamma\leq\text{PSL}_2(\cO)$ consists of $P\in\bH$ above $\overline{F}$ that lie outside or on $S_g$ for every $g\in\Gamma$. Figure \ref{fig:1} shows a hollowed out Ford domain of $\text{PSL}_2(\bZ[i\sqrt{10}])$ from above. The imaginary axis runs left to right.

\begin{theorem}\label{thm:intro}The only Euclidean hemispheres that contribute a face to the Ford domain of $\emph{PE}_2(\cO)$ are those of radius 1 centered on elements of $\cO$.\end{theorem}

The projective elementary group always contains the reflection and shifts \begin{equation}\label{eq:2}r=\begin{bmatrix}0 & -1\\ 1 & 0\end{bmatrix}\hspace{0.5cm}\text{and}\hspace{0.5cm}s(\alpha)=\begin{bmatrix}1 & \alpha\\ 0 & 1\end{bmatrix}\end{equation} for $\alpha\in \cO$. In particular, $rs(-\alpha)\in\text{PE}_2(\cO)$, which has isometric hemisphere with radius $1$ and center $\alpha$. Theorem \ref{thm:intro} asserts there is nothing else for the Ford domain of $\text{PE}_2(\cO)$. We could turn Figure \ref{fig:1} into a Ford domain for $\text{PE}_2(\bZ[i\sqrt{10}])$ by erasing all hemispheres that lie completely below the white, translucent plane.

Let $\Delta$ denote the discriminant of $\cO$. When $|\Delta| > 12$ there is a gap between rows of unit hemispheres. The infinitude of points $\lambda/\mu$ with $(\lambda,\mu)=\cO$ that lie in this gap provides a new proof of Nica's theorem \cite{nica} \cite{senia}, which is a strengthened form of Cohn's theorem \cite{cohn}.

\begin{corollary}\label{cor:intro1}For $|\Delta| > 12$, $\emph{PE}_2(\cO)$ is an infinite-index subgroup of $\emph{PSL}_2(\cO)$. Furthermore, $\emph{PE}_2(\cO)$ is its own normalizer.\end{corollary}

(The proofs of Theorem \ref{thm:intro} and Corollary \ref{cor:intro1} appear as the proof of Theorem 3.1.5 in the author's dissertation \cite{martin}.)

Notably, if $\cO$ is any ring of integers other than non-Euclidean, imaginary quadratic \cite{vaser} or if $n$ is anything other than $2$ \cite{bass}, $\text{PE}_n(\cO)=\text{PSL}_n(\cO)$.

Nica proves non-normality of $\text{PE}_2(\cO)$ but does not compute the normalizer. Also, the approach in this paper is simpler. There is no need for the ``special unimodular pairs" or solutions to Pell-type equations used in \cite{nica}. 

Stange has also reproved the infinite-index assertion of Corollary \ref{cor:intro1} by way of her Schmidt arrangements \cite{stange}. And another proof of both the infinite-index and non-normal assertions has been given recently in \cite{senia2}. Sheydvasser computes a fundamental polyhedron for the group generated by integer translations and reflections over unit hemispheres in $\bH$ centered on integers from an imaginary quadratic field or certain quaternion algebras. By showing that this group is commensurable with $\text{PSL}_2(\cO)$, Sheydvasser recovers Nica's theorem over a larger set of rings. 

A Ford domain for $\text{PE}_2(\cO)$ also gives a group presentation via Poincar\'{e}s polyhedron theorem. Generators are face-pairing matrices and relations are defined by reflections and edge cycles. (See Section 6.2 of \cite{fine2} for a description and \cite{maskit} for a proof.) This produces the following presentation of $\text{PE}_2(\cO)$, originally due to Fine \cite{fine} (derived largely from Cohn's main lemma in \cite{cohn2}).

\begin{corollary}\label{cor:intro2}Let $\tau=\sqrt{\Delta}/2$ or $(1+\sqrt{\Delta})/2$ depending on $\Delta\,\emph{mod}\,2$. If $|\Delta| > 12$, $$\emph{PE}_2(\cO)=\big\langle r, s(1), s(\tau)\;\big|\;s(1)s(\tau)s(1)^{-1}s(\tau)^{-1},\,r^2,\,(rs(1))^3\big\rangle.$$\end{corollary}

Finally, we use the Ford domain of $\text{PE}_2(\cO)$ to prove that $\text{PSL}_2(\cO)$ factors nontrivially as a free product with amalgamation. In 1970 \cite{serre}, Serre showed that $\text{PSL}_2(\cO)$ does not have the property FA as defined in \cite{serre2} (see the comment just before Chapter 2). For finitely generated groups, not satisfying FA is strictly weaker than being a nontrivial amalgam, and it was conjectured that the stronger property holds (as discussed in Chapter 1 of \cite{fine2}). This was confirmed by Fine and Frohman in 1988. It is a result which Fine describes in Section 1.3 of \cite{fine2} as answering one of the foremost questions addressed in his research monograph. Their proof relies on the theory of fundamental groups of factor manifolds. They find a nicely embedded, incompressible, separating two-sided surface in the factor manifold of $\text{PSL}_2(\cO)$, which gives an amalgam splitting by the Seifert-Van Kampen theorem. 

We are able to avoid such complexity. Matrices that pair faces lying at least partially above the Euclidean plane $t=2/3$ (shown translucent, white in Figure \ref{fig:1}) generate $\text{PE}_2(\cO)$ by Theorem \ref{thm:intro}. Those from faces lying at least partially below $t=2/3$ form a group we call $\Gamma$. It is straightforward to check that the resulting amalgam is nontrivial.

\begin{corollary}\label{cor:intro3}If $|\Delta| > 12$, then $\emph{PSL}_2(\cO)\simeq \emph{PE}_2(\cO)*_{N}\Gamma$ is a nontrivial amalgam, where $N$ is generated by $r$, $s(1)$, and $s(\tau)rs(\tau)^{-1}$ and $\Gamma$ depends on $\cO$.\end{corollary}

This is the same amalgam splitting given in \cite{frohman}.

Remark that \cite{fine} and \cite{frohman} only mention Corollaries \ref{cor:intro1} and \ref{cor:intro2} in the context of maximal orders, but it appears that their proofs apply without modification when $\cO$ is a non-maximal order.

\begin{figure}
    \centering
    \includegraphics[width=0.936\textwidth]{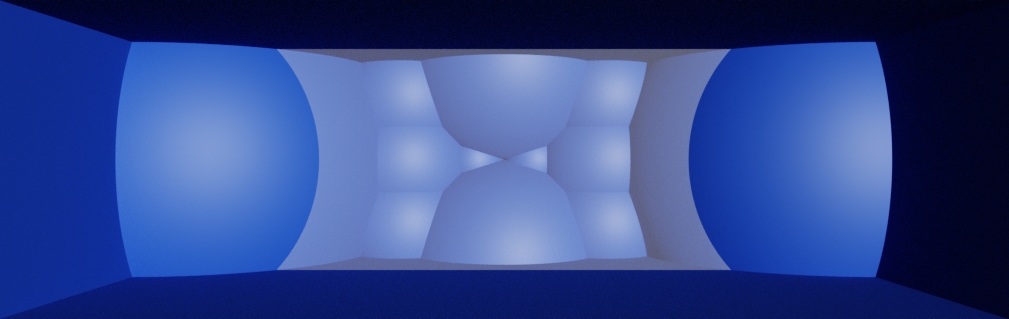}
    \captionsetup{width=0.936\textwidth}\caption{Amalgam splitting of $\text{PSL}_2(\bZ[i\sqrt{10}])$ with the plane $t=2/3$.}
    \label{fig:1}
\end{figure}

\section{Computing the Ford domain}\label{sec:2}

We do not contribute to the study of $\text{PE}_2(\cO)$ or $\text{PSL}_2(\cO)$ when $|\Delta|\leq 12$, the cases where the two groups coincide. Our aim is to compute $\cF$ for $\text{PE}_2(\cO)$ when $|\Delta|>12$. Two lemmas are needed. The first is well-known and straightforward to check.

\begin{lemma}\label{lem:mag}If $|\Delta| > 12$ and $\alpha\in\cO$ is not $0$, $1$, or $-1$, then $|\alpha|\geq 2$.\qed\end{lemma}

Recall notation in (\ref{eq:2}). In all studies of $\text{PE}_2(\cO)$ of which the author is aware, its elements are written in the form $s(\alpha_n)r\cdots s(\alpha_1)r$ where $\alpha_2,...,\alpha_{n-1}\neq 0,1,-1$. This may have originated in \cite{cohn} as Nica calls it Cohn's standard form. It differs slightly, though nontrivially, from Lemma \ref{lem:form} below. Our proof of Theorem \ref{thm:intro} does not work when Cohn's form is used in place of the following.

\begin{lemma}\label{lem:form}If $|\Delta| > 4$, every matrix in $\emph{PE}_2(\cO)$ is equal to a product of the form $s(\alpha_n)r\cdots s(\alpha_1)rs(\alpha_0)$ for some $\alpha_0,...,\alpha_n\in\cO$ with $\alpha_1,...,\alpha_{n-1}\neq 0,1,-1$.\end{lemma}

\begin{proof}If $|\Delta| > 4$, the only units in $\cO$ are $1$ and $-1$. In particular, $s(\alpha)$ for $\alpha\in\cO$ account for all upper triangular matrices in $\text{PSL}_2(\cO)$, and $rs(\alpha)r$ account for all lower triangular matrices. Thus $r$ and $s(\alpha)$ generate $\text{PE}_2(\cO)$ for $\alpha\in\cO$ (or just $\alpha$ from a $\bZ$-basis for $\cO$). Moreover, since $r$ has order two and $s(\alpha)s(\beta)=s(\alpha+\beta)$, any matrix in $\text{PE}_2(\cO)$ can be written as $s(\alpha_n)r\cdots s(\alpha_1)rs(\alpha_0)$ for some $\alpha_0,...,\alpha_n\in\cO$. To see that $\alpha_i\neq 0,1,-1$ is possible if $0<i<n$, observe that $s(\alpha_{i+1})rs(0)rs(\alpha_{i-1})=s(\alpha_{i+1}+\alpha_{i-1})$ when $\alpha_i=0$, and $s(\alpha_{i+1})rs(\pm1)rs(\alpha_{i-1}) = s(\alpha_{i+1}\mp 1)rs(\alpha_{i-1}\mp 1)$ when $\alpha_i=\pm 1$.\end{proof}

We can now prove Theorem \ref{thm:intro}.

\begin{theorem}\label{thm:main}The only Euclidean hemispheres that contribute a face to the Ford domain of $\emph{PE}_2(\cO)$ are those of radius $1$ centered on elements of $\cO$.\end{theorem}

\begin{proof}Let $\zeta\in\bC$ lie outside each closed unit disc centered on an integer in $\cO$. The claim follows if we can show $\zeta$ is not under any isometric hemisphere from $\text{PE}_2(\cO)$. 

Fix $g\in\text{PE}_2(\cO)$ with nonzero bottom-left entry. By Lemma \ref{lem:form}, we may write $g=s(\alpha_n)r\cdots s(\alpha_1)rs(\alpha_0)$ for some $\alpha_0,...,\alpha_n\in\cO$ with $\alpha_1,...,\alpha_{n-1}\neq 0,1,-1$. Let $\zeta_1=\zeta+\alpha_0$. Note that $\zeta_1$ lies outside the unit disc centered on $0$, so $|\zeta_1|>1$. For $i=1,...,n-1$, let $\zeta_{i+1}=\alpha_i-1/\zeta_i$. This gives \begin{equation}\label{eq:3}s(\alpha_i)r\begin{bmatrix}\zeta_i\\ 1\end{bmatrix}=\zeta_i\begin{bmatrix}\zeta_{i+1}\\ 1\end{bmatrix}.\end{equation} Assume $|\zeta_i|>1$ for induction. Then $\zeta_{i+1}$ lies inside the open unit disc centered on $\alpha_i$. Since $\alpha_i\neq 0,1,-1$, this implies $|\zeta_{i+1}|>1$ by Lemma \ref{lem:mag}.

Let $-\beta$ and $\alpha$ be the bottom-row entries of $g$. Our goal is to show $|\alpha/\beta-\zeta| > 1/|\beta|$. First observe that $\beta\neq 0$ forces $n\geq 1$. Next, the bottom-row entries of $s(\alpha_n)r$ are $1$ and $0$, so $$\alpha-\beta\zeta=\begin{bmatrix}-\beta & \alpha\end{bmatrix}\begin{bmatrix}\zeta \\ 1\end{bmatrix}=\begin{bmatrix}1 & 0\end{bmatrix}s(\alpha_{n-1})r\cdots s(\alpha_1)rs(\alpha_0)\begin{bmatrix}\zeta \\ 1\end{bmatrix}=\zeta_n\cdots \zeta_1,$$ where the last equality uses (\ref{eq:3}). Since $|\zeta_n\cdots\zeta_1|>1$, we are done.\end{proof}

\begin{corollary}\label{cor:nica}For $|\Delta| > 12$, $\emph{PE}_2(\cO)$ is an infinite-index subgroup of $\emph{PSL}_2(\cO)$. Furthermore, $\emph{PE}_2(\cO)$ is its own normalizer.\end{corollary}

\begin{proof}Let $g\in\text{PSL}_2(\cO)$ have top- and bottom-left entries $\lambda$ and $\mu$. If $\lambda/\mu$ is not under the closure of any isometric hemisphere of $\text{PE}_2(\cO)$, then $g$ is the unique matrix in $\text{PE}_2(\cO)M$ that has minimal bottom-left entry magnitude $|\mu|$. Since there are infinitely many such $\lambda/\mu$ when $|\Delta|>12$, there are infinitely many right cosets.

The same idea proves the stabilizer claim: Take any $g\in\text{PSL}_2(\cO)$ that is not in $\text{PE}_2(\cO)$, with $\lambda$ and $\mu$ as before. Assume without loss of generality that $\lambda/\mu$ is not under any (open) isometric hemisphere from $\text{PE}_2(\cO)$. The top- and bottom-left entries of $gs(\alpha)g^{-1}$ are $-\alpha\lambda\mu+1$ and $-\alpha\mu^2$. Their ratio is $\lambda/\mu-1/\alpha\mu^2$, which is not under any isometric hemisphere of $\text{PE}_2(\cO)$ for the appropriate choice of $\alpha$. Thus $g\text{PE}_2(\cO)g^{-1}\neq\text{PE}_2(\cO)$.\end{proof}

Next we recover Fine's presentation of $\text{PE}_2(\cO)$ \cite{fine}.

\begin{corollary}\label{cor:fine}Let $\tau=\sqrt{\Delta}/2$ or $(1+\sqrt{\Delta})/2$ depending on $\Delta\,\emph{mod}\,2$. If $|\Delta| > 12$, $$\emph{PE}_2(\cO)=\big\langle r, s(1), s(\tau)\;\big|\;s(1)s(\tau)s(1)^{-1}s(\tau)^{-1},\,r^2,\,(rs(1))^3\big\rangle.$$\end{corollary}

\begin{proof}Let $F$ from Definition \ref{def:bianchi} be the Voronoi cell around $0$ in the lattice $\cO$---a rectangle when $\Delta$ is even and a hexagon when $\Delta$ is odd---so that only the hemisphere centered at $0$ contributes a face of $\cF$. Then the face-pairing matrices are just $r$, which self-pairs the hemisphere, and $s(1)$ and $s(\tau)$, which pair parallel vertical walls. When $\Delta$ is even, the relations are $r^2$ from the self-paired hemisphere, $s(1)s(\tau)s(1)^{-1}s(\tau)^{-1}$ from the edge cycle of length four among the four vertical walls, and $(rs(1))^2$ from the edge cycle of length two where the hemisphere meets the two vertical walls $\Re(\zeta)=\pm1/2$. The only difference when $\Delta$ is odd occurs among the vertical edges, of which there are now six. They split into two cycles of length three. Both cycles give the same relation, $s(1)s(\tau)s(1)^{-1}s(\tau)^{-1}$.\end{proof}

We are now ready to prove that $\text{PSL}_2(\cO)$ is a free product with amalgamation.

\begin{corollary}\label{cor:frohman}If $|\Delta| > 12$, then $\emph{PSL}_2(\cO)\simeq \emph{PE}_2(\cO)*_{N}\Gamma$ is a nontrivial amalgam, where $N$ is generated by $r$, $s(1)$, and $s(\tau)rs(\tau)^{-1}$ and $\Gamma$ depends on $\cO$.\end{corollary}

\begin{proof}Take $F$ from Definition \ref{def:bianchi} to be the rectangle centered at $\sqrt{\Delta}/4$, and let $\cF$ be a Ford domain for $\text{PSL}_2(\cO)$ as shown in Figure \ref{fig:1}. Write $\text{PSL}_2(\cO)=\langle G\,|\,R\rangle$ using Poincar\'{e}s polyhedron theorem. By Theorem \ref{thm:main}, the faces and edges that lie at least partially above the Euclidean plane $t=2/3$ define a presentation of $\text{PE}_2(\cO)$, say $\langle G_a\,|\,R_a\rangle$. Let $G_b\subset G$ be the generators from faces that lie at least partially below $t=2/3$, and let $R_b$ be the relations among them from $R$. Since paired faces are either both at least partially above or both at least partially below $t=2/3$, we have $G=G_a\cup G_b$. The same is true of edges in a cycle, so $R=R_a\cup R_b$. 

Let $N$ denote the subgroup of $\text{PE}_2(\cO)$ generated by $G_a\cap G_b$, which consists of $s(1)$ to pair the two vertical walls that intersect $t=2/3$, $r$ to self-pair the unit hemisphere centered at $0$, and $s(\tau)rs(\tau)^{-1}$ to self-pair or pair the remaining unit hemisphere or hemispheres depending on $\Delta\,\text{mod}\,2$. If necessary, add relations to $R_b$ so that the inclusion map $N\hookrightarrow\langle G_b\,|\,R_b\rangle$ is well-defined. (It is not necessary, but proving this is not useful.) Then $\text{PSL}_2(\cO)=\langle G_a\cup G_b\,|\,R_a\cup R_b\rangle \simeq \text{PE}_2(\cO) *_N\langle G_b\,|\,R_b\rangle$. Let us show $N\neq\text{PE}_2(\cO)$ to verify that our amalgam is nontrivial. By Corollary \ref{cor:fine}, the map $r,s(1)\mapsto\text{Id}$ and $s(\tau)\mapsto s(\tau)$ extends to a homomorphism $\text{PE}_2(\cO)\to \langle s(\tau)\rangle$ since all three group relations map to the identity. This homomorphism is surjective with kernel containing $N$.\end{proof}

\section*{Statements and declarations}

\subsection*{Conflict of interest} There are no conflicts of interest to report.

\subsection*{Funding} This research received no external funding.

\subsection*{Author contribution} The author confirms sole responsibility for the proofs herein and for the preparation of this manuscript.

\subsection*{Data availability} The Ford domain data for $\text{PSL}_2(\bZ[i\sqrt{10}])$ were produced by John Cremona's implementaion of Swan's algorithm \cite{cremona}. No other data were used.

\bibliographystyle{plain}
\bibliography{refs}

\end{document}